\documentclass[article]{IEEEtran}
\usepackage{cite}
\usepackage{color}
\usepackage{graphicx}
\graphicspath{{fig/}{jpeg/}}
\usepackage[cmex10]{amsmath}
\usepackage{amssymb}
\usepackage{algorithm}
\usepackage{algpseudocode}

\usepackage{amsmath,amssymb,lipsum}
\usepackage{hyperref}
\usepackage{comment}
\usepackage{graphicx}
\usepackage[font=small]{caption}
\usepackage{subcaption}
\usepackage{array}
\input{mysymbol.sty}
\usepackage{needspace}

\usepackage{amsthm,bm}
\usepackage{tikz}
\usetikzlibrary{shapes,arrows}

\newtheorem{proposition}{Proposition}

\newtheorem{theorem}{Theorem}
\newtheorem{definition}{Definition}
\newtheorem{lemma}{Lemma}
\newtheorem{corollary}{Corollary}

\theoremstyle{definition}
\newtheorem{assumption}{Assumption}
\newtheorem{remark}{Remark}

\title{Distributed Constrained Online Learning} 
\author{Santiago Paternain$^\dagger$, Soomin Lee$^*$, Michael M. Zavlanos$^{\S}$ and Alejandro Ribeiro$^\dagger$
  \thanks{Work supported by ARL DCIST CRA W911NF-17-2-0181.

    $^\dagger$  Dept. of Electrical and System Engineering, Univ. of Pennsylvania. Email: \{spater, aribeiro\} @seas.upenn.edu.

    $^*$ Yahoo! Research. Email: soominl@yahoo-inc.com 

    $^\S$Department of Mechanical Engeeniring and Material Science, Duke University. Email:michael.zavlanos@duke.edu
}
}

\begin{document}

\maketitle
\thispagestyle{empty}
\pagestyle{empty}


%
\begin{abstract}

In this paper, we consider groups of agents in a network that select actions in order to satisfy a set of constraints that vary arbitrarily over time and minimize a time varying function of which they have only local observations. The selection of actions, also called a strategy, is causal and decentralized, i.e., the dynamical system that determines the actions of a given agent depends only on the constraints at the current time and on its own actions and those of its neighbors. To determine such a strategy, we propose a decentralized saddle point algorithm and show that the corresponding global fit and regret are bounded by functions of the order of $\sqrt{T}$. Specifically, we define the global fit of a strategy as a vector that integrates over time the global constraint violations as seen by a given node. The fit is a performance loss associated with online operation as opposed to offline clairvoyant operation which can always select an action, if one exists, that satisfies the constraints at all times. If this fit grows sublinearly with the time horizon it suggests that the strategy approaches the feasible set of actions. Likewise, we define the regret of a strategy as the difference between its accumulated cost and that of the best fixed action that one could select knowing beforehand the time evolution of the objective function. Numerical examples support the theoretical conclusions.
\end{abstract}
%
%

\section{Introduction}
Distributed optimization has applications in several engineering problems such as source localization \cite{rabbat2004distributed}, resource allocation problems in multi cellular communication networks \cite{shen2012distributed}, machine learning \cite{john2010elements,cavalcante2009adaptive}, multi-robot teams \cite{koppel2015d4l,koppel2018decentralized,freundlich2018distributed} and the internet of things \cite{ghosh2015pricing,gatsis2017wireless}.  In these problems agents try to optimize collectively a common objective function that is separable in local objectives. In cases where a centralized solution of such problems is acceptable and the objectives and constraints are convex, the problem reduces to solve a classic convex optimization problem. Several methods to do so exist, notably the saddle point algorithm by Arrow and Hurwicz \cite{arrow_hurwicz}, which has the advantage of admitting a distributed implementation \cite{zhu2012distributed,yuan2011distributed,chang2014distributed,feijer2010stability}. In this framework the most studied classes of problems are when the constraints and the objective function are constant with respect of the time and when their variation is according to a stationary probability distribution. Solutions to the former problem have been established \cite{zhu2012distributed,yuan2011distributed,chang2014distributed,feijer2010stability}. In the latter, the problem that is studied is that of selecting actions that minimize the expectation of the objective while satisfying the constraints in average. When the problem is unconstrained, centralized and decentralized implementations of stochastic gradient descent converge to such solutions \cite{robbins1951stochastic,schmidt2017minimizing}.

In this paper we consider online formulations in which the cost and constraints can vary arbitrarily over time, even strategically. In this case, cost minimization can be formulated in the language of regret \cite{blackwell1956analog,Zinkevich03,vapnik1995nature} whereby agents select online actions that result in a cost chosen by nature. The cost functions are revealed to the agents after the action are selected and these values are used to adapt the future strategy. Regret is defined as the difference between the total cost incurred by each agent and the cost of the optimal fixed solution that a clairvoyant agent could select. In that sense, it measures the effect of not knowing the temporal evolution of the cost function and, therefore, it can be used as a performance measure for online operation.  Likewise, the fit of a strategy \cite{PaternainRibeiro16,chen2017online} is a vector that integrates over time the constraint violation incurred by each agent. In that sense the fit is a performance loss associated with online operation as opposed to offline clairvoyant operation which can always select an action that satisfies the constraints at all times, if such action exists. It is a remarkable fact that online versions of the Arrow-Hurwicz algorithm for centralized problems achieve regret bounded by a constant and fit bounded by a function that grows at a sublinear rate \cite{PaternainRibeiro16,chen2017online,mahdavi2012trading}; this suggests vanishing per play penalties and constraint violation of online plays with respect to clairvoyant agent. Since the fit compares the accumulation of the constraints it is possible to achieve small fit by alternating periods in which the constraints are satisfied with slack and periods in which the constraints are violated. This notion is appropriated for constraints that have a cumulative nature. However, when this is not the case it is possible to work with the notion of saturated fit, where only violations of the constraints are accumulated. A trajectory with small saturated fit is one in which the constraints are violated by a significant amount only for a short period of time.

The problem of distributed online constrained convex optimization has been studied in \cite{hosseini2013online,lee2016distributed2,lee2016distributed}.  Continuous time approaches have advantages in the context of distributed control systems whenever signals are inherently continuous. Moreover, discrete-time approaches can be characterized as the discretization of continuous-time dynamics.  The work in \cite{hosseini2013online,lee2016distributed2} considers unconstrained settings where agents minimize a time-varying convex loss, whereas here we consider constrained problems. These have been considered in \cite{lee2016distributed} using a Saddle Point algorithm establishing sublinear bounds on the network disagreement and on the regret achieved by the strategy. The main difference with \cite{lee2016distributed} is that instead of imposing exact consensus among agents, here we allow for small disagreement. This idea has been used for unconstrained problems in\cite{koppel2016proximity}. With this modification we are able to establish that the trajectories that arise from the distributed online saddle point dynamics are such that the disagreement of agents, the regret and the fit are bounded by sublinear functions of the time horizon (Section \ref{sec_analysis}). This suggests that our proposed algorithm achieves consensus, feasibility and optimality as the time goes to infinity.  Central to the development of this result are the definitions of global fit and regret (Section \ref{sec_problem_formulation}). The former is defined as a vector that contains the time integrals of the constraints of all the agents evaluated across the trajectory of a specific agent. Having small global fit means that the agent is able to satisfy the constraints of every other agent, and thus if it is placed in a different position in the network its performance is maintained.  Likewise, the global regret is the evaluation of the accumulated global cost of a specific agent's trajectory with respected to the optimal centralized clairvoyant solution.

To solve the constrained online optimization problem under consideration, we propose an online distributed saddle point algorithm to control the growth of the global fit and regret. Saddle point algorithms update the primal variables -- the actions -- along the negative of a weighted linear combination of the gradients of each constraints. Since the feasible set of actions is typically represented by the intersection of the sub level sets of the constraints functions, this linear combination pushes the actions towards feasibility. The weights of this linear combination are updated according to the current violation of the constraints. If an action violates a constraint by much, its corresponding weight is increased faster, whereas if a constraint is satisfied its corresponding weight is reduced.  We start by showing that if an action that is feasible for all agents and for all times exists, then the network disagreement is bounded by a function that is sublinear with respect to the time horizon. Based on this result we establish sublinear global fit and global regret. We also illustrate our algorithm on a problem involving a team of robots driving through an urban environment to perform real-time texture classification for the purpose of mapping and object recognition. We show that the team of robots succeds in training a common classifier that allows them to distinguish between grass and pavement images even when some of the agents have only observed one of the classes.
The rest of this paper is organized as follows. In Section \ref{sec_problem_formulation} we formalize the online distributed optimization problem. Later we present the Distributed Online Saddle Point Algorithm (Section \ref{sec_algorithm}) and we establish that it achieves consensus, feasibility and optimality in Section \ref{sec_analysis}. Other than concluding remarks the paper finishes with numerical examples in section \ref{sec_numerical}.

 

%
\section{Constrained Online Learning in Networks}\label{sec_problem_formulation}

We consider a group of $N$ agents linked by an undirected connected graph $\ccalG = \left\{\ccalV,\ccalE\right\}$ where $\ccalV = \{ 1, \ldots, N \}$ is a set of nodes and $\ccalE$ is a set of edges so that $(i,j) \in \ccalE$ means that $i$ and $j$ are connected to each other. The set $\ccalN_i:=\{j:(i,j) \in \ccalE\}$ contains all nodes that are connected to $i$ and is called the neighborhood of $i$. Note that since the graph is undirected, node $j$ is in the neighborhood of $i$ if and only if node $i$ is in the neighborhood of $j$. 

We are interested in situations where the agents in $\ccalG$ have access to arbitrarily time varying local constraints and local objective functions and continuously select actions that are good not only for their local constraints and costs but for the local constraints and costs of other agents. To explain this formally, let $t\in\reals^+$ be a continuous time index, $f_{i}(t,\cdot): \reals^n\to \reals^{m_i}$ be a set of $m_i$ convex constraints at agent $i$ and $f_{0i}(t,\cdot): \reals^n\to \reals$ be a local convex cost incurred at node $i$. A {\it local} goal of node $i$ is to select an action $\bbx_i\in\ccalX\subseteq\reals^n$ that satisfies local constraints $f_{i}(t,\bbx_i)\preceq 0$ across a time interval $[0,T]$ while minimizing the cost $f_{0i}(t,\bbx_i)$ integrated over the same interval. This is tantamount to defining the locally optimal solution $\bbx_i^\ell$ over the interval $[0,T]$ as
\begin{alignat}{2}\label{eqn_local_optimization_problem}
   \bbx_i^\ell := & \argmin_{\bbx_i\in \ccalX} \ && \int_0^T f_{0i}(t,\bbx_i) \;dt \nonumber \\ 
                     & \st                   \ && f_i(t,\bbx_i)\preceq 0, 
                                             \    \forall\ t\in [0,T].
\end{alignat}
Problem \eqref{eqn_local_optimization_problem} models situations in which each of the agents is acting independently since the optimal action of $i$ depends on its local cost and constraints, and not on those of other agents. Instead, here we are interested in situations where the actions of agents are coordinated, so that an action $\bbx_i$ of agent $i$ can affect the costs and constraints of other agents. This results in a {\it global} formulation in which the optimal action of each agent $\bbx_i^*$ is defined as the one that satisfies the constraints of all agents and minimizes the integral of the sum cost, 
\begin{alignat}{2}\label{eqn_global_optimization_problem}
   \bbx_i^*\ :=\ &\argmin_{\bbx_i\in \ccalX} \
                   &&     \int_0^T \sum_{j=1}^N f_{0j}(t,\bbx_i) \;dt , \nonumber\\
                 & \st \                   
                   &&     f_j(t,\bbx_i)\preceq 0, \quad
                          \forall\ j \text{\ and\ } t\in [0,T].
\end{alignat}
We say that the problem in \eqref{eqn_global_optimization_problem} is global because the action $\bbx_i$ is evaluated at the constraint and costs of all nodes. This readily implies that $\bbx_i^* = \bbx_j^*$ and that there is a single global action that is optimal for all nodes. For future reference we define the sum cost function $f_{0}(t,\bbx_i):=\sum_{j=1}^N f_{0j}(t,\bbx_i)$ and the aggregate constraint $f(t,\bbx_i) := [f_1(t,\bbx_i); \ldots; f_N(t,\bbx_i)]$. With this notation, the problem in \eqref{eqn_global_optimization_problem} can be rewritten as  
\begin{alignat}{2}\label{eqn_global_optimization_problem_simplified_notation}
   \bbx_i^* \ =\  & \argmin_{\bbx_i\in \ccalX}  && \int_0^T  f_{0}(t,\bbx_i) \;dt, \nonumber\\
                  & \st                \ && f(t,\bbx_i)\preceq 0, \quad
                                        \forall\ t\in [0,T].
\end{alignat}
While \eqref{eqn_local_optimization_problem} models agents acting in isolation, \eqref{eqn_global_optimization_problem} and \eqref{eqn_global_optimization_problem_simplified_notation} model agents acting in concert. The latter situation arises when  local functions are related to a common variable. E.g., the costs and constraints can represent local observations of a parameter to be estimated \cite{jakubiec2013d} or local observations and costs of a plant to be controlled \cite{nedic2010constrained}. Problems having the form of \eqref{eqn_global_optimization_problem_simplified_notation} also arise in large scale optimization where costs and constraints are {\it not} acquired locally but are distributed over several servers to reduce computation and storage \cite{braun1996collaborative}.

If the functions $f_{0i}(t,\cdot)$ and $f_{i}(t,\cdot)$ are available for all times $t\in [0,T]$, solving \eqref{eqn_global_optimization_problem_simplified_notation} reduces to solving a  distributed convex optimization problem for which a number of standard algorithms are applicable; see e.g., \cite{boyd2011distributed,nedic2015distributed}. In this paper we consider problems in which the constraints $f_i(t,\cdot)$ and costs $f_0(t,\cdot)$ are arbitrary and observed causally and locally by node $i$. In this setting it makes sense to consider time varying strategies $\bbx_i(t)$ that adapt the action of agent $i$ to the information that is revealed at time $t$. In this context the optimal argument in \eqref{eqn_global_optimization_problem_simplified_notation} is a clairvoyant action that would be chosen when agents have knowledge of the future evolution of the system at time $t=0$. The appropriate figures of merit in this case are the notions of regret \cite{Zinkevich03, shalev2011online, hazan2007logarithmic} and fit \cite {PaternainRibeiro16} that we generalize to network settings in the following section. {These quantities compare the performance of the online distributed operation with the offline centralized solution of \eqref{eqn_global_optimization_problem_simplified_notation}.}

Before proceeding with definitions of network regret and fit, notice that for the definition in \eqref{eqn_global_optimization_problem_simplified_notation} to be valid the function $f_0(t,\bbx)$ has to be integrable with respect to the time variable $t$. In subsequent definitions and analyses we further require the network to be connected and the constraints $f_i$ for all $i\in\ccalV$, to be integrable, convex and Lipschitz continuous with respect to $\bbx$ for all times $t$.  We formally state these assumptions next.
%
%
\begin{assumption}\label{assumption_network}
The network is connected with diameter $D$, i.e., the shortest distance between the two most distant nodes in the network is $D$. 
\end{assumption}
%
\begin{assumption}\label{assumption_convexity}
Let $\ccalX$ be a compact convex set and the functions $f_{0i}(t,\bbx)$ and $f_i(t,\bbx)$ be integrable with respect to $t$ and convex with respect to $\bbx\in\ccalX$ for all $t\in [0,T]$. We further assume that cost and constraints are Lipschitz continuous over $\ccalX$ with respective constants $L_0>0$ and $L_f>0$. I.e., for any $\bbx,\bby \in \ccalX$ and all $t\in[0,T]$ the cost functions satisfy 
\begin{equation}
   \big|f_{0i}(t,\bbx)-f_{0i}(t,\bby)\big| \leq L_0\big\|\bbx-\bby\big\|,
\end{equation}
and the constraint functions satisfy
\begin{equation}
   \big|f_{kj}(t,\bbx)-f_{kj}(t,\bby)\big| \leq L_f\big\|\bbx-\bby\big\|,
\end{equation}
where $f_{kj}(t,\cdot)$ denotes the $j$th component of the vector valued constraint function $f_{k}(t,\cdot)$.
\end{assumption}

%
We remark that integrability with respect to $t$ is a weak condition. We do not require differentiability, not even continuity. This entails a fundamental difference with time varying optimization problems that strive to track a time varying optimal argument under the assumption of smooth time varying costs and constraints \cite{jakubiec2013d,simonetto2016class,fazlyab2018prediction}. The goal here is to design an algorithm that can adapt to unexpected changes in the system, including, indeed, most importantly, to those that arise because of discontinuities in the cost and constraint functions. Another requirement for $\bbx_i^*$ to be well defined is existence of an action $\bbx^\dagger \in \ccalX$ that satisfies the constraints at all times and all nodes as we formally state next.

%
\begin{assumption}\label{assumption_viability}
There exists an action $\bbx^\dagger \in \ccalX$ that satisfies the constraints of all agents for all times $t\in[0,T]$,
\begin{equation}
   f_i(t,\bbx^\dagger)\prec 0,  \qquad \forall\ i \text{\ and\ } t\in[0,T].
\end{equation}
We say that $\ccalX^{\dagger} := \left\{ \bbx^\dagger \in \ccalX: f_i(t,\bbx^\dagger)\preceq 0,  \forall i \text{\ and\ } t\in[0,T].\right\}$ is the set of feasible actions.
\end{assumption}
%
%
We require as well that minimum of the objective function does not become progressively smaller with time so that a uniform bound $K$ holds for all times $t\in[0,T]$. 
The existence of the bound in \eqref{eqn_lower_bound} is a mild requierement. Since the function $f_0(t,\bbx)$ is convex, for any time $t$ it is lower bounded for compact set of actions $\ccalX$. The only restriction imposed is that $\min_{\bbx \in \ccalX^N} f_0(t,\bbx)$ does not become progressively smaller with time so that a uniform bound $K$ holds for all times $t\in[0,T]$.

\begin{assumption}\label{assumption_bound}
  There exists $K>0$ independent of the time horizon $T$ such that for all $t\in [0,T]$ it holds that
  \begin{equation}\label{eqn_lower_bound}
f_0(t,\bbx^*)-\min_{x\in \ccalX^N} f_0(t,\bbx) \leq K,
  \end{equation}
  where $\bbx^*$ is the solution to the problem \eqref{eqn_global_optimization_problem_simplified_notation}.
  \end{assumption}
%
%
%
\subsection{Network Regret and Network Fit}  
 %
%
To evaluate the cost performance of such trajectories we define the notions of network regret and network fit. Begin then by considering a trajectory $\bbx_i(t)$ chosen by agent $i$ and the total accumulated cost $\int_0^T f_{0j}(t,\bbx_i(t))dt$ that this trajectory incurs for 
a possibly different agent $j$. We define the regret $\ccalR_{Tj}^i$ as the difference of this accumulated cost relative to the corresponding cost that would be incurred by the optimal trajectory $\bbx_i^*$ of \eqref{eqn_global_optimization_problem_simplified_notation},
\begin{equation}\label{eqn_network_regret}
   \ccalR_{Tj}^i := \int_0^T f_{0j}(t,\bbx_i(t) ) \,dt 
                          -\int_0^T f_{0j}(t,\bbx_i^*  ) \,dt.
\end{equation}
Likewise, we consider the accumulation $\int_0^T f_{j}(t,\bbx_i(t))dt$ of constraint of agent $j$ incurred by the trajectory of agent $i$. The fit $\ccalF_{Tj}^i$ is defined as the comparison of this constraint accumulation relative to the corresponding constraint accumulation of the optimal trajectory $\bbx_i^*$
\begin{equation}\label{eqn_network_fit}
  \ccalF_{Tj}^i := \int_0^T f_{j}(t,\bbx_i(t) ) \,dt 
                          -\int_0^T f_{j}(t,\bbx_i^*  ) \,dt.
\end{equation}
The action $\bbx_i^*$ can be considered as an offline reference that would be chosen by an entity that is clairvoyant, because it observes the future, and omniscient, because it observes the costs and constraints of all nodes. Our objective is to consider trajectories that are chosen online by agents that are causal, because they observe the past, and local, because they observe their local costs and exchange information with neighboring nodes only. In this context regret and fit can be interpreted as performance losses associated with online causal and local operation as opposed to offline clairvoyant and omniscient operation. If $\ccalF_{Tj}^i$ is positive we are in a situation in which, had the constraints of all agents be known beforehand, we could have selected an action $\bbx^\dagger$ to satisfy all constraints. Due to its cumulative nature, it is possible to achieve small (local) fit by alternating between actions for which the constraints take positive and negative values. This is an appropriate model for quantities that can be stored -- such as energy budgets enforced through average power constrains. However, in other settings this formulation can be a limitation. This drawback can be overcome by defining the saturated fit in which constraints slacks are saturated to a small negative constant. We discuss this in Section \ref{sec_saturated_fit}.

The fit measures how far the trajectory $\bbx(t)$ is from achieving that goal. Analogously, if the regret $\ccalR_{T j}^i$ is large we are in a situation in which prior knowledge of the objective functions and constraints would had resulted in the selection of an strategy $\bbx^*$ that achieves much better performance than the one achieved by $\bbx_i(t)$. In that sense $\ccalR_{T j}^i$ indicates how much we regret not having had that information a priori. 

A good learning strategy is one that achieves small regret and fit as that would be an indication that the trajectory $\bbx(t)$ approaches $\bbx^*$. Notice however that since the objective function and the constraints are integrated over a time horizon $T$, it is natural to expect the cost and constraints to grow linearly with $T$. Thus, having regret and fit that grow at a sublinear rate is sufficient indication of a good learning strategy. This intuition motivates the following definitions of feasible and optimal trajectories.

\begin{definition}\label{def_feasibility_optimality}
We define an environment as a set of constraints $f_j:\reals\times \reals^n \to \reals^{m_j}$ and costs $f_{0j}:\reals\times \reals^n \to \reals$ for all $j\in\ccalV$. For a trajectory $\bbx_i(t)$ we consider the regret and fit definitions in \eqref{eqn_network_regret} and \eqref{eqn_network_fit} and further define the sum regret $\ccalR_{T}^i := \sum_{j\in\ccalV}\ccalR_{Tj}^i$ and the network wide fit $\ccalF_T^j = [ \ccalF_{T1}^{j\top}, \ldots, \ccalF_{TN}^{j\top}  ]^\top$. We say that:

\begin{mylist}
\item[\bf Feasibility.] The trajectories are feasible in the environment if all the local fits $\ccalF_{T}^i$ with $i\in\ccalV$ grow sublinearly with $T$. I.e., there exist a function $h(T)$ with $\limsup_{T\to\infty} h(T)/T = 0$  and a constant vector $C_f$ such that for all times $T$ it holds,
\begin{equation}\label{eqn_feasiibility}
   \ccalF_{T}^i := \int_0^T f(t,\bbx_i(t)) \,dt \leq C_f h(T).
\end{equation} 
\item[\bf Optimality.] The trajectories are optimal in the environment if all regrets $\ccalR_{T}^i$ grow sublinearlly for all $i\in\ccalV$ and $T$. I.e. there exist a function $h(T)$ with $\limsup_{T\to\infty} h(T)/T = 0$ and a constant $C$ such that for all times $T$ it holds,
\begin{equation}\label{eqn_optimality}
   \ccalR_{T}^i := \int_0^T f_0(t,\bbx_i(t)) \,dt - \int_0^T f_0(t,\bbx^*) \,dt  \leq Ch(T).
\end{equation}
\end{mylist} \end{definition}
In the next section we develop the details of a distributed and online version of the Arrow-Hurwicz algorithm, such that its generated trajectories are feasible and optimal in the sense of Definition \ref{def_feasibility_optimality}. The latter is formally stated and proved in Section \ref{sec_analysis} along with an intermediate result that claims that the disagreement across agents is sublinear with respect to the time horizon, hence suggesting consensus. Before doing so, we make a pertinent remark highligthing the differences between the results in this work and those achieved for the centralized algorithm.
\begin{remark}\label{rmk_existing_literature}
    The centralized version of the problem here considered \eqref{eqn_global_optimization_problem_simplified_notation}, can be solved -- when affordable -- by an online saddle-point algorithm \cite{PaternainRibeiro16}. The trajectories that arise from such algorithm achieve (i) regret bounded by a constant independent of the time horizon $T$ for the unconstrained problem; (ii) fit bounded by a constant independent of the time horizon $T$ if the objective function is constant with respect to the action, i.e., feasibility problems and (iii) regret bounded by a constant independent of the time horizon $T$ and fit bounded by a sublinear function of $T$ otherwise. By considering the distributed version of the previous algorithm we cannot establish regret or fit bounded by constants independent of $T$ because of the disagreement across agents.  
\end{remark}
%


\section{Distributed Online Saddle point}\label{sec_algorithm}
{
The problem defined in \eqref{eqn_global_optimization_problem_simplified_notation} is a centralized optimization problem in which all agents should select the same action. Since each agent $i\in\ccalV$ has access only to the local cost and constraints, a more natural representation of the problem \eqref{eqn_global_optimization_problem_simplified_notation} is one where each agent selects a local decision vector $\bbx_i\in\mathbb{R}^n$.} Nodes then try to achieve the minimum of their local objective functions $f_{0i}(t,\bbx_i)$ while satisfying the local constraints $f_i(t,\bbx_i)\preceq 0$ and keeping their variables equal to the variables $\bbx_j$ of neighboring nodes $j\in\ccalN_i$. By defining $\bbx = [\bbx_1^\top,\ldots, \bbx_N^\top]^\top$, this alternative formulation can be written as 
\begin{alignat}{2}\label{eqn_consensus_problem}
   \bbx^* := & \argmin_{\bbx\in \ccalX^N} \ 
                  && \int_0^T f_{0}(t,\bbx) \;dt =\argmin_{\bbx\in \ccalX^N}                    \int_0^T \sum_{i=1}^Nf_{0i}(t,\bbx_i) \;dt\nonumber\\
             & \st \ 
                  &&  f_i(t,\bbx_i)\preceq 0, \forall t\in [0,T], \forall i\in\ccalV, \\
             &    &&  \bbx_i = \bbx_j, \forall i\in\ccalV,j \in \ccalN_i.\nonumber
\end{alignat}
Since the network is assumed to be connected (cf., Assumption \ref{assumption_network}), the constraints $\bbx_i=\bbx_j$ for all $i$ and $j\in\ccalN_i$ imply that \eqref{eqn_global_optimization_problem_simplified_notation} and \eqref{eqn_consensus_problem} are equivalent. The previous problem 
can solved in a distributed manner with a variety of methods, one of which is he saddle point algorithm of Arrow and Hurwicz \cite{arrow_hurwicz}. In this work we aim to extend the saddle point algorithm to control the growth of regret and fit. In doing so it is convenient to relax the consensus constraints $\bbx_i = \bbx_j$ in \eqref{eqn_consensus_problem} to allow for some controlled disagreement. We accomplish this by defining the set of constraints
\begin{equation}
g_{ij}(\bbx_i,\bbx_j) = \left\|\bbx_i-\bbx_j\right\|^2-\gamma \leq 0, 
\end{equation}
where $\gamma$ is a positive constant limiting how much constraint violation is allowed. Notice that the parameter $\gamma$ could be set to be arbitrarily small. The advantage of using a controlled disagreement is that it allows for agents to achieve a good global performance without damaging excessively the local performance, which in some applications might be important as well. By allowing larger values of $\gamma$, we allow more disagreement and therefore we prioritize the local performance, whereas by making $\gamma$ closer to zero the goal is set in the centralized performance. The same relaxation is considered in \cite{koppel2016proximity} in the case of unconstrained distributed optimization. The proximity constraints allows to write the problem of interest in the following form
\begin{alignat}{2}\label{eqn_optimization_problem}
   \tbx^* := & \argmin_{\bbx\in \ccalX^N} \ 
                  && \int_0^T f_{0}(t,\bbx) \;dt \nonumber\\
             & \st \ 
                  &&  f_i(t,\bbx_i)\preceq 0, \forall t\in [0,T], \forall i\in\ccalV, \\
             &    &&  g_{ij}(\bbx_i,\bbx_j)= \left\|\bbx_i - \bbx_j\right\|^2-\gamma\leq 0, \forall i,j\in\ccalV. \nonumber
\end{alignat}
For the above online optimization problem we can construct the following time varying Lagrangian
\begin{equation}\label{eqn_lagrangian}
{\ccalL}(t,\bbx,\bm{\lambda},\bm{\mu}) = f_0(t,\bbx)+\sum_{i=1}^N\left(\bm{\lambda}_i^\top f_i(t,\bbx_i)+\bm{\mu}_i^\top g_i(\bbx)\right),
  \end{equation}
where $\bm{\lambda}_i\in\mathbb{R}_+^{m_i}$ for $i=1\ldots N$ and $\bm{\mu}_i\in\mathbb{R}_+^{|\ccalN_i|}$ for $i=1\ldots N$ are the Lagrange multipliers and where $g_i(\bbx)\in\mathbb{R}^{|\ccalN_i|}$ is the vector with components $g_{ij}(\bbx_i,\bbx_j)$ for all $j\in \ccalN_i$. Saddle point methods rely on the fact that for a constrained convex optimization problem, a pair is a primal-dual solution if and only if it is a saddle point of the Lagrangian associated with the problem, see e.g. \cite{boyd2004convex}. This is the case in problem \eqref{eqn_optimization_problem} since $f_0(\cdot,\bbx)$, $f(\cdot,\bbx)$ and $g(\cdot,\bbx)$ are convex (c.f. Assumption \ref{assumption_convexity}). In addition, because $\bm{\lambda},\bm{\mu} \succeq 0$, the Lagrangian is convex with respect to $\bbx$ and therefore the subgradient with respect to $\bbx$  exists for all time $t\geq 0$, let us denote it by $\ccalL_x(t,\bbx,\bm{\lambda},\bm{\mu})$. The Lagrangian is linear with respect to $\bm{\lambda}$ and $\bm{\mu}$ and therefore its partial derivatives with respect to these variables exist. Let us denote them by $\ccalL_{\lambda}(t,\bbx,\bm{\lambda},\bm{\mu})$ and $\ccalL_{\mu}(t,\bbx,\bm{\lambda},\bm{\mu})$ respectively. The actions $\bbx$ are updated -- as in the classic Arrow-Hurwicz algorithm -- by following the negative subgradient of the Lagrangian with respect to $\bbx$ 
\begin{equation}\label{eqn_grad_descent_basic}
\begin{split}
\dot{\bbx} &= -\ccalL_x(t,\bbx,\bm{\lambda},\bm{\mu}) \\&= - f_{0,x}(t,\bbx) - \sum_{i=1}^N f_{i,x}(t,\bbx_i)^\top \bm{\lambda}_i 
- \sum_{i=1}^N \sum_{j\in\ccalN_i} \mu_{ij} g_{ij,x}(\bbx),
\end{split}
  \end{equation}
where $\ccalN_i$ is the set of neighbors of node $i$. The primal update interprets the constraints as a potentials with corresponding weights $\bm{\lambda}$ and $\bm{\mu}$ and descends along a linear combination of the gradients of said potentials. The multipliers are then updated by following the subgradient of the Lagrangian with respect to them
\begin{subequations}\label{eqn_grad_ascent_basic}
  \begin{equation}
\dot{\bm{\lambda}} = \ccalL_{\lambda}(t,\bbx,\bm{\lambda},\bm{\mu}) = f(\bbx),
  \end{equation}
    \begin{equation}
\dot{\bm{\mu}} = \ccalL_{\mu}(t,\bbx,\bm{\lambda},\bm{\mu}) = g(\bbx).
    \end{equation}
  \end{subequations}
The intuition behind the latter update is that if a constraint is violated, for instance $f_{1,1}(\bbx) >0$ the corresponding multiplier, $\bm{\lambda}_{1,1}$ will be increased, thus augmenting the relative weight of this potential in the linear combination in \eqref{eqn_grad_descent_basic}. Which in turn pushes the action towards satisfying said constraint. On the other hand, if the constraint is satisfied, the weight of that potential will be reduced, thus making the direction of the gradient of the function less important in the weighted linear combination. Observe that the multipliers need to remain positive at all time to ensure the convexity of the Lagrangian with respect to $\bbx$, yet if a multiplier takes the value zero and its corresponding constraint is satisfied, the previous update turns the multiplier negative. To avoid this issue, we will require a projection over the positive orthant. We formalize this idea next, after making the observation that the update \eqref{eqn_grad_descent_basic}--\eqref{eqn_grad_ascent_basic} is indeed distributed. To see this, write the Lagrangian as a sum of the following local Lagrangians
\begin{equation}
  \begin{split}
    \ccalL^i(t,\bbx,\bm{\lambda},\bm{\mu}) &=  f_0^i(t,\bbx_i) +\bm{\lambda}_i^\top f_i(t,\bbx_i) \\
    &+ \sum_{j\in\ccalN_i}{\mu}_{ij}\left(\left\|\bbx_i-\bbx_j\right\|^2-\gamma\right),
\end{split}
  \end{equation}
where to compute each local Lagrangian, agent $i$ needs only information regarding its variables and those of its neighbors. Then, each agent can compute locally the gradient of the Lagrangian with respect to its local variable $\bbx_i$ and perform the update described in \eqref{eqn_grad_ascent_basic}--\eqref{eqn_grad_descent_basic}, with the caveat that to ensure that the multipliers are always positive we need to consider a projected dynamical system. We formalize this idea next.  
%
\begin{definition}[\bf Projection of a vector at a point]
  Let $K\subset \mathbb{R}^n$ be a compact convex set. Then, for any $\bby\in K$ and $\bbv\in\mathbb{R}^n$, we defined the projection of $\bbv$ over the set $K$ at the point $\bby$ as
  \begin{equation}
\Pi_K\left[\bby,\bbv\right] = \lim_{\xi\to 0^+} \frac{P_K(\bby+\xi\bbv)-\bby}{\xi},
    \end{equation}
  where the standard projection $P_K(\bbz) = \argmin_{\bby\in K}\|\bby-\bbz \|^2$ is always well defined because $K$ is convex.
  \end{definition}
%
The intuition behind the projection is that, if the point $\bby$ is in the interior of the set $K$ then the projection of the vector $\bbv$ is the vector itself. In cases where $\bby$ is in the boundary of the set $K$, the projection of $\bbv$ is its component tangental to the boundary of $K$. With this definition at hand, and by defining the gain of the controller to be $\varepsilon>0$ we define the distributed online saddle point controller as follows. Each agent updates its action by following the negative subgradient of the Lagrangian with respect to its local copy of the action $\bbx_i$ 
 \begin{equation}\label{eqn_gradient_descent}
\dot{\bbx}_i = \Pi_{\ccalX}\left[\bbx_i, -\varepsilon\ccalL_{x_i}(t,\bbx,\bm{\lambda},\bm{\mu})\right],
\end{equation}
Likewise, the multipliers $\bm{\lambda}_i$ and $\bm{\mu}_i$ are updated by ascending along the direction of the gradient of the Lagrangian with respect to $\bm{\lambda}$ and $\bm{\mu}$ respectively, i.e., 
\begin{subequations}\label{eqn_gradient_ascent}
\begin{equation}\label{eqn_gradient_ascent_lambda}
\dot{\bm{\lambda}}_i = \Pi_{+}\left[\bm{\lambda}_i,\varepsilon\ccalL_{\lambda_i}(t,\bbx,\bm{\lambda},\bm{\mu})\right],
\end{equation}
\begin{equation}\label{eqn_gradient_ascent_mu}
\dot{{\mu}}_{ij} = \Pi_{+}\left[\mu_{ij},\varepsilon\ccalL_{\mu_{ij}}\left(t,\bbx,\bm{\lambda},\bm{\mu}\right)\right].
\end{equation}
\end{subequations}
The three gradients can be computed in a distributed fashion since they only depend on each agent's own variables and those of their neighbors. In \cite{PaternainRibeiro16} it was shown that in the centralized case, a saddle point algorithm such as the one described by \eqref{eqn_gradient_descent} and \eqref{eqn_gradient_ascent} achieves feasible and strongly optimal trajectories, i.e., fit bounded by a sublinear function of the time horizon and regret bounded by function that is constant with respect to the time horizon. In this work we show that the distributed version of said algorithm (c.f. \eqref{eqn_gradient_descent} and \eqref{eqn_gradient_ascent}) achieves feasible and optimal trajectories in the sense of Definition \ref{def_feasibility_optimality}. Moreover, the network disagreement is bounded by a function that is sublinear with respect to the time horizon. These results are the subject of the next section. 


\section{Feasible and Optimal trajectories}\label{sec_analysis}
Let us consider an energy-like function which will be used in subsequent analysis. Let $\tilde{\bbx}\in \ccalX^N$, $\tilde{\bm{\lambda}} \in \mathbb{R}_+^{\sum_i m_i}$, $\tilde{\bm{\mu}}\in\mathbb{R}_{+}^{\sum_i |\ccalN_i|}$, where we denote by $|\ccalN_i|$ the cardinality of the set of neighbors of node $i$, and define the function 
\begin{equation}\label{eqn_energy_function}
V_{\tilde{\bbx},\tilde{\bm{\lambda}},\tilde{\bm{\mu}}}(\bbx,\bm{\lambda},\bm{\mu}) = \frac{1}{2}\left( \|\bbx-\tilde{\bbx}\|^2 +\|\bm{\lambda}-\tilde{\bm{\lambda}}\|^2 + \|\bm{\mu}-\tilde{\bm{\mu}}\|^2\right).
\end{equation}
By considering the time derivative of the previous function along the dynamics \eqref{eqn_gradient_descent}--\eqref{eqn_gradient_ascent} we establish that the integral of the difference of the Lagrangian evaluated at $(\bbx(t), \tilde{\bm{\lambda}},\tilde{\bm{\mu}})$ and the Lagrangian evaluated at $(\tilde{\bbx}, \bm{\lambda}(t), \bm{\mu}(t))$ is bounded by a constant independent of the time horizon $T$. The following lemma -- key to establish that the saddle point dynamics yield feasible and optimal trajectories -- formalizes this result.
%
\begin{lemma}\label{lemma_key_lemma}
    Let Assumptions \ref{assumption_network}--\ref{assumption_viability} hold. Then for any $T\geq 0$ the solutions of the dynamical system \eqref{eqn_gradient_descent}--\eqref{eqn_gradient_ascent} satisfy
  \begin{equation}\label{eqn_key_lemma}
  \begin{split}
    \int_0^T\ccalL(t,\bbx(t),\tilde{\bm{\lambda}},\tilde{\bm{\mu}})& -\ccalL(t,\tilde{\bbx},\bm{\lambda}(t),\bm{\mu}(t))\, dt\\
    &\leq  \frac{{V}_{\tilde{\bbx},\tilde{\bm{\lambda}},\tilde{\bm{\mu}}}(\bbx(0),\bm{\lambda}(0),\bm{\mu}(0))}{\varepsilon},
  \end{split}
  \end{equation}
  where ${V}_{\tilde{\bbx},\tilde{\bm{\lambda}},\tilde{\bm{\mu}}}(\bbx(0),\bm{\lambda}(0),\bm{\mu}(0))$ is the energy-function defined in \eqref{eqn_energy_function} eavluated values for the actions and multipliers and time zero and $\tilde{\bbx},\tilde{\bm{\lambda}}, \tilde{\bm{\mu}}$ are arbitrary.
\end{lemma}
\begin{proof}
  See Appendix \ref{ap_key_lemma}.
\end{proof}
%
%
By analyzing the expression \eqref{eqn_key_lemma} for different choices of $\tilde{\bbx}, \tilde{\bm{\lambda}}$ and $\tilde{\bm{\mu}}$ it is possible to establish that the saddle point dynamics \eqref{eqn_gradient_descent}--\eqref{eqn_gradient_ascent} yields sublinear network disagreement for all $T>0$. We formalize this result in Proposition \ref{proposition_disagreement}. 
\begin{proposition}[\bf Sublinear Network Disagreement]\label{proposition_disagreement}
  Let Assumptions \ref{assumption_network}--\ref{assumption_bound} hold. Then for any $T\geq 0$ the solutions of the dynamical system \eqref{eqn_gradient_descent}--\eqref{eqn_gradient_ascent} are such that the network disagreement is sublinear with respect to $T$. In particular for $\bm{\lambda}(0) =0$ and $\bm{\mu}(0)=0$, for any $i,j \in\ccalV$ we have that
  \begin{equation}\begin{split}
      &\int_0^T\left\|\bbx_i(t)-\bbx_j(t)\right\| \, dt \\
      &\leq D\sqrt{(K+\gamma)T +\frac{1}{2\varepsilon}\left(1+\left\|\bbx^*-\bbx(0)\right\|^2\right)},
      \end{split}
    \end{equation}
    where $D$ is the network diameter defined in Assumption \ref{assumption_network}.
  \end{proposition}
\begin{proof}
  Let us consider the expression \eqref{eqn_key_lemma} with $\tilde{\bbx}= \bbx^*$, the solution of the problem \eqref{eqn_optimization_problem}, $\tilde{\bm{\lambda}} = \bm{0}$, and $\tilde{{\mu}}_{ij} = 1$ for some $i\in\ccalV$ and $j\in\ccalN_i$ and $\tilde{\bm{\mu}}_{ik} = {0}$ for all $k\neq j$ and $\tilde{\bm{\mu}_{l}} = \bm{0}$ for all $l\neq i$. For this selection of variables the Lagrangian evaluated at $(t,\bbx(t),\tilde{\bblambda},\tilde{\bbmu})$ yields
  \begin{equation}
\ccalL(t,\bbx(t),\tilde{\bblambda},\tilde{\bbmu})=f_0(\bbx(t)) +\left(\left\|\bbx_i(t) - \bbx_j(t)\right\|^2-\gamma\right). 
    \end{equation}
  Applying Lemma \ref{lemma_key_lemma} for this particular choice of $\tilde{\bbx},\tilde{\bblambda}$ and $\tilde{\bbmu}$, \eqref{eqn_key_lemma} reduces to
  \begin{equation}
  \begin{split}
    \int_0^T f_0(\bbx(t)) +\left(\left\|\bbx_i(t) - \bbx_j(t)\right\|^2-\gamma\right)\, dt \\
    -\int_0^T\left(f_0(\bbx^*) +\bm{\lambda}(t)^\top f(\bbx^*)\right)\,dt\\
    - \int_0^T \sum_{i=1}^N\sum_{j\in\ccalN_i}{{\mu}}_{ij}(t)\left(\left\|\bbx^*_i-\bbx_j^*\right\|^2-\gamma\right) \, dt\\
    \leq  \frac{{V}_{\tilde{\bbx},\tilde{\bm{\lambda}},\tilde{\bm{\mu}}}(\bbx(0),\bm{\lambda}(0),\bm{\mu}(0))}{\varepsilon}.
  \end{split}
    \end{equation}
  Since $\bbx^*$ is the solution to \eqref{eqn_consensus_problem} it holds that $\bbx^*_i =\bbx_j^*$ and that $f(\bbx^*) \preceq \bm{0}$. Since $\bm{\lambda}(t)$ and $\bm{\mu}(t)$ are always in the positive orthant, due to the projection in their update (c.f. \eqref{eqn_gradient_ascent_lambda} and \eqref{eqn_gradient_ascent_mu}), we have that $\bm{\lambda}(t)^\top f(\bbx^*) \leq 0 $ and that ${\mu}_{ij}(t) \left(\left\|\bbx_i^*-\bbx_j^* \right\|^2-\gamma\right) =-\mu_{ij}(t)\gamma \leq 0$ for all $t\in[0,T]$. These observations imply that 
 \begin{equation}\label{eqn_disagreement_1}
  \begin{split}
    \int_0^T \left(f_0(\bbx(t))-f_0(\bbx^*) + \left\|\bbx_i(t) - \bbx_j(t)\right\|^2-\gamma\right)\,dt \\
    \leq  \frac{{V}_{\tilde{\bbx},\tilde{\bm{\lambda}},\tilde{\bm{\mu}}}(\bbx(0),\bm{\lambda}(0),\bm{\mu}(0))}{\varepsilon}.
\end{split}.
 \end{equation}
 From the definition of a minimum and Assumption \ref{assumption_bound} it follows that 
 \begin{equation}\label{eqn_bound_objective_function}
f_0(t,\bbx(t))-f_0(t,\bbx^*) \geq \min_{x\in \ccalX^N} f_0(t,\bbx)-f_0(t,\bbx^*) \geq -K.
 \end{equation}
Substituting \eqref{eqn_bound_objective_function} into \eqref{eqn_disagreement_1} yields 
\begin{equation}
  \begin{split}
    \int_0^T -(K+\gamma) + \left\|\bbx_i(t) - \bbx_j(t)\right\|^2 \,dt \\
    \leq  \frac{{V}_{\tilde{\bbx},\tilde{\bm{\lambda}},\tilde{\bm{\mu}}}(\bbx(0),\bm{\lambda}(0),\bm{\mu}(0))}{\varepsilon}.
  \end{split}
    \end{equation}
    Rearranging the terms in the previous expression it holds that
 \begin{equation}\label{eqn_bound_disagreement_3}
   \begin{split}
\int_0^T\left\|\bbx_i(t) - \bbx_j(t)\right\|^2 \,dt  
\leq  (K+\gamma)T\\
+\frac{{V}_{\tilde{\bbx},\tilde{\bm{\lambda}},\tilde{\bm{\mu}}}(\bbx(0),\bm{\lambda}(0),\bm{\mu}(0))}{\varepsilon}.
    \end{split}
 \end{equation}
Because the square function is convex, by virtue of Jensen's inequality we have that  
\begin{equation}
       \left(\int_0^T \left\|\bbx_i(t) - \bbx_j(t)\right\|\, dt\right)^2 \leq \int_0^T\left\|\bbx_i(t) - \bbx_j(t)\right\|^2 \,dt.  
\end{equation} 
The previous inequality provides a lower bound for \eqref{eqn_bound_disagreement_3}, hence we have that 
   \begin{equation}
     \begin{split}
       \left(\int_0^T \left\|\bbx_i(t) - \bbx_j(t)\right\|\, dt\right)^2 &\leq (K+\delta)T\\
       &+\frac{{V}_{\tilde{\bbx},\tilde{\bm{\lambda}},\tilde{\bm{\mu}}}(\bbx(0),\bm{\lambda}(0),\bm{\mu}(0))}{\varepsilon}.
\end{split}
     \end{equation}
By taking the square root of the previous inequality we observe that the disagreement among neighbors is sublinear. In particular, by evaluating ${V}_{\tilde{\bbx},\tilde{\bm{\lambda}},\tilde{\bm{\mu}}}(\bbx(0),\bm{\lambda}(0),\bm{\mu}(0))$ for the selection of $\tilde{\bbx}$, $\tilde{\bm{\lambda}}$ and $\tilde{\bm{\mu}}$ done at the begining of the proof and for the initial conditions $\bm{\lambda}(0)=\bm{0}$, $\bm{\mu}(0)=\bm{0}$ yields
   \begin{equation}
     \begin{split}
 &\int_0^T \left\|\bbx_i(t) - \bbx_j(t)\right\|\, dt  \\
      & \leq\sqrt{(K+\delta)T+\frac{1}{2\varepsilon}\left(1+\left\|\bbx^*-\bbx(0)\right\|^2\right)}.
\end{split}
     \end{equation}
   The latter establishes a sublinear disagreement among one hop neighbors on the network. To show that the result holds for every pair of nodes use the triangular inequality and the fact that the diameter of the network is $D$ (Assumption \ref{assumption_network}). 
\end{proof}
The sublinear network disagreement that the previous proposition establishes, along with the result of Lemma \ref{lemma_key_lemma}, allows us to prove that the local Fit and Regret are bounded by sublinear functions with respect to the time horizon $T$. The latter means that the trajectories that arise from the Distributed Online Saddle Point Dynamics \eqref{eqn_gradient_descent}--\eqref{eqn_gradient_ascent} are feasible and optimal in the sense of definition \ref{def_feasibility_optimality}. We formalize these results in theorems \ref{theo_main} and \ref{theo_regret} respectively. 
%
\begin{theorem}[\bf Feasibility]\label{theo_main}
Let Assumptions \ref{assumption_network}--\ref{assumption_bound} hold. Then for any $T\geq 0$ the solutions of the dynamical system \eqref{eqn_gradient_descent}--\eqref{eqn_gradient_ascent}, with $\varepsilon>1/2$, are such that the $k$-th component of the local fit $\ccalF_{T j}^{i}$ for any $i,j\in\ccalV$ is bounded by
    %
$\left(\ccalF_{T j}^{i}\right)_k \leq O(\sqrt{T})$.
    %
\end{theorem}
\begin{proof}
  We evaluate the expression \eqref{eqn_key_lemma} for the particular choice of $\tilde{\bbx}=\bbx^*$ and $\tilde{\bbmu}=0$
  \begin{equation}\label{eqn_key_eq_appendix_pre}
    \begin{split}
      \int_0^Tf_0(t,\bbx(t)) - f_0(t,\bbx^*)dt \\
      \sum_{i=1}^N \int_0^T\left[\tilde{\bm{\lambda}}_i^\top f_i(t,\bbx_i(t))-\bm{\lambda}_i^\top(t)f_i(t,{\bbx}^*_i)\right] \,dt\\
      -\sum_{i=1}^N \sum_{j\in \ccalN_i} \int_0^T\mu_{ij}(t)\left(\left\|{\bbx}^*_i-{\bbx}^*_j\right\|^2 -\gamma\right) \, dt  \\ 
      \leq \frac{1}{\varepsilon}V_{{\bbx}^*,\tilde{\bm{\lambda}},{\bm{0}}}(\bbx(0),{\bm{\lambda}}(0),\bm{\mu}(0)).
\end{split}
    \end{equation}
  By virtue of Assumption \ref{assumption_bound} and the bound derived in \eqref{eqn_bound_objective_function}, we can lower bound the difference of the integrals of the objective functions by 
  \begin{equation}
\int_0^Tf_0(t,\bbx(t)) - f_0(t,\bbx^*)dt \geq -\int_0^T Kdt = -KT.
\end{equation}
  Substituting the previous bound in \eqref{eqn_key_eq_appendix_pre} yields 
\begin{equation}\label{eqn_key_eq_appendix}
\begin{split}
      \sum_{i=1}^N \int_0^T\left[\tilde{\bm{\lambda}}_i^\top f_i(t,\bbx_i(t))-\bm{\lambda}_i^\top(t)f_i(t,{\bbx}^*_i)\right] \,dt\\
      -\sum_{i=1}^N \sum_{j\in \ccalN_i} \int_0^T\mu_{ij}(t)\left(\left\|{\bbx}^*_i-{\bbx}^*_j\right\|^2 -\gamma\right) \, dt  \\ 
      \leq \frac{1}{\varepsilon}V_{{\bbx}^*,\tilde{\bm{\lambda}},{\bm{0}}}(\bbx(0),{\bm{\lambda}}(0),\bm{\mu}(0)) + KT.
\end{split}
    \end{equation}
Observe that, since $\bbx^*$ is the solution of the decentralized problem \eqref{eqn_consensus_problem},  we have that $f_i(t,{\bbx}^*_i) \preceq \bm{0}$. Moreover, because the Lagrange Multipliers $\bm{\lambda}_i(t)$ are in the positive orthant (cf., \eqref{eqn_gradient_ascent_lambda}) the product $-\bm{\lambda}_i^\top(t)f_i(t,{\bbx}_i^*)$ is always positive. Likewise, the product $\bbmu_{ij}(t)\left(\gamma-\left\|\bbx_i^*-\bbx_j^*\right\|^2\right)=\bbmu_{ij}(t)\gamma \geq 0$, for all $t\geq 0$. With these considerations the following bound holds 
\begin{equation}\label{eqn_key_eq_appendix}
\begin{split}
      \sum_{i=1}^N \int_0^T\tilde{\bm{\lambda}}_i^\top f_i(t,\bbx_i(t)) \,dt
      \leq \frac{1}{\varepsilon}V_{\tilde{\bbx},\tilde{\bm{\lambda}},\tilde{\bm{\mu}}}(\bbx(0),{\bm{\lambda}}(0),\bm{\mu}(0)) + KT.
\end{split}
    \end{equation}
  Denote by  $\left[\cdot \right]^+$ the projection over the positive orthant. Then, for a particular $i\in\ccalV$ choose $\tilde{\bm{\lambda}}_i = \left[\ccalF_{T i}^{i} \right]^+$ and $\tilde{\bm{\lambda}}_j= \bm{0}$ for all $j\in\ccalV$ such that $j\neq i$. Then, \eqref{eqn_key_eq_appendix} reduces to 
\begin{equation}\label{eqn_first_fit_bound}
\begin{split}
      \left\| \left[\ccalF_{Ti}^{i}\right]^+\right\|^2
      \leq \frac{1}{\varepsilon}V_{{\bbx}^*,{\tilde{\bm{\lambda}}},{\bm{0}}}(\bbx(0),{\bm{\lambda}}(0),\bm{\mu}(0)) + KT.
\end{split}
    \end{equation}
  Without loss of generality assume $\bm{\lambda}(0) = 0$ and $\bm{\mu}(0)=0$. Then, the right hand side of the previous expression reduces to
  \begin{equation}
    V_{{\bbx}^*,{\tilde{\bm{\lambda}}},{\bm{0}}}(\bbx(0),{\bm{\lambda}}(0),\bm{\mu}(0))= \frac{\left\|\bbx(0)-\bbx^*\right\|^2+      \left\| \left[\ccalF_{Ti}^{i}\right]^+\right\|^2}{2},
  \end{equation}
  and \eqref{eqn_first_fit_bound} can be written as
  \begin{equation}
\left(1-\frac{1}{2\varepsilon}\right)\left\| \left[\ccalF_{Ti}^{i}\right]^+\right\|^2
      \leq \frac{1}{2\varepsilon}\left(\left\|\bbx(0)-\bbx^*\right\|^2\right) + KT.
  \end{equation}
  Then, for any $\varepsilon>1/2$, the previous expression yields
  \begin{equation}
    \left\| \left[\ccalF_{Ti}^{i}\right]^+\right\|^2 \leq \frac{\left(\left\|\bbx(0)-\bbx^*\right\|^2+2\varepsilon KT \right)}{2\varepsilon-1}.
    \end{equation}
  Taking the square root of both sides of the previous inequality shows that the norm of the projection of the fit is bounded by a sublinear function. In particular we have that each component $k=1\ldots m_i$ of $\ccalF_{Ti}^{i}$ is also upper bounded by the a sublinear function that grows as $O(\sqrt{T})$. We are left to show that for any $j \neq i$ we have fit that is bounded by a function of the order of $\sqrt{T}$. The latter is a consequence of the Liptchitz continuity and the sublinear network disagreement as we show next. Add and subtract $f_{i,k}(t,\bbx_i(t))$ to the definition of the local fit to write
\begin{equation}
  \begin{split}
    \left(\ccalF_{Tj}^{i}\right)_k = \int_0^T f_{i,k}(t,\bbx_j(t)) dt    = \int_0^T f_{i,k}(t,\bbx_i(t)) dt \\
    + \int_0^T f_{i,k}(t,\bbx_j(t)) -f_{i,k}(t,\bbx_i(t))dt.
    \end{split}
\end{equation}
The first term on the right hand side of the previous expression is, by definition, $\left(\ccalF_{Ti}^{i}\right)_k$. This term is bounded by a function of the order as $\sqrt{T}$ as previously shown. On the other hand, the second integral can be bounded using the Liptchitz continuity of $f_{i,k}(t,\bbx)$ (c.f. Assumption \ref{assumption_convexity}) by 
\begin{equation}
  \begin{split}
  \left|\int_0^T f_{i,k}(t,\bbx_j(t)) -f_{i,k}(t,\bbx_i(t))dt\right| \\
  \leq \int_0^TL_f\|\bbx_i(t)-\bbx_j(t)\|dt 
\end{split}
  \end{equation}
Then, the result of Proposition \ref{proposition_disagreement} completes the proof.
  \end{proof}
%
The previous theorem establishes that the local fit achieved by a system that follows saddle point dynamics \eqref{eqn_gradient_descent}--\eqref{eqn_gradient_ascent} is bounded by a function whose rate of growth is sublinear, thus suggesting vanishing penalties. However, as discussed previously this can be achieved by solutions that oscillate, i.e., trajectories that violate the constraints at some times and that satisfy them with slack at other periods. Since this is not desirable in some applications, we overcome this limitation by showing that the saturated global fit has the same property. We address this in Section \ref{sec_saturated_fit}.  The fact that the fit grows sublinearly is equivalent to achieving trajectories that are feasible in the sense of Definition \ref{def_feasibility_optimality}. In the next theorem we establish that the local regrets are also bounded and thus, that the saddle point dynamics give origin to trajectories that are also optimal. 
\begin{theorem}[\bf Optimality]\label{theo_regret}
Let Assumptions \ref{assumption_network}--\ref{assumption_bound} hold. Then for any $T\geq 0$ the solutions of the dynamical system \eqref{eqn_gradient_descent},\eqref{eqn_gradient_ascent_lambda} and \eqref{eqn_gradient_ascent_mu} are such that the local regret $\ccalR_T^{i}$ for any $i\in\ccalV$ is bounded by a function of $O(\sqrt{T})$. In particular, for $\bblambda(0) =0$ and $\bbmu(0) =0$ we have that 
\begin{equation}\label{eqn_bound_regret}
  \begin{split}
    &\ccalR_T^{i} \leq \frac{\left(1+\left\|\bbx^*-\bbx(0)\right\|^2\right)}{\varepsilon}\\
   & +(N-1)L_0 D \sqrt{(K+\gamma)T +\frac{1}{2\varepsilon}\left(1+\left\|\bbx^*-\bbx(0)\right\|^2\right)}
    \end{split}
      \end{equation}
\end{theorem}
\begin{proof}
  Let us consider the local regret of agent $j$, with $j\in\ccalV$, defined in \eqref{eqn_optimality}
  \begin{equation}
\ccalR_T^j = \int_0^T \sum_{i=1}^N \left(f_0^i(t,\bbx_j(t)) - f_0^i(t,\bbx^*_i)\right) \, dt. 
    \end{equation}
  Add and subtract $\sum_{i=1, i\neq j}^Nf_0^i(t,\bbx_i(t))$ to the previous equation and rewrite the previous expression as
  \begin{equation}\label{eqn_first_regret_bound}
    \begin{split}
    \ccalR_T^j &= \int_0^T f_0(t,\bbx(t)) - f_0(t,\bbx^*) \, dt \\
    &+\int_0^T \sum_{i=1, i\neq j}^N f_0^i(t,\bbx_j(t)) - f_0^i(t,\bbx_i(t)) \, dt.
    \end{split}
    \end{equation}
  From Lemma \ref{lemma_key_lemma} and by choosing $\tilde{\bbx}=\bbx^*,\tilde{\bm{\lambda}}=0$, $\tilde{\bm{\mu}}=0$ it follows that
  \begin{equation}\begin{split}
    \int_0^T \sum_{i=1}^N \left(f_0^i(t,\bbx_i(t)) - f_0^i(t,\bbx^*_i) \right)\, dt  \\
     -\sum_{i=1}^N \int_0^T\bm{\lambda}_i^\top(t)f_i(t,{\bbx}^*_i) \,dt\\
      -\sum_{i=1}^N \sum_{j\in \ccalN_i} \int_0^T\bbmu_{ij}(t)\left(\left\|{\bbx}^*_i-{\bbx}^*_j\right\|^2 -\gamma\right) \, dt  \\ 
      \leq \frac{1}{\varepsilon}V_{\bbx^*,\bm{0},\bm{0}}\left(\bbx(0),\bm{\lambda}(0), \bm{\mu}(0)\right).
      \end{split}
    \end{equation}
  As it was previously argued, because $\bbx^*$ is the solution of \eqref{eqn_consensus_problem} it follows that $f_i(\bbx^*)\preceq 0$ and that $\bbx_i^*=\bbx_j^*$. Combining these observations with the fact that the multipliers always lie in the positive orthants due to the projections introduced in their updates \eqref{eqn_gradient_ascent}, we have that
$-\bblambda_i(t)^\top f_i(t,\bbx_i^*)\geq 0 $ and $\bbmu_{ij}(t)\left(\left\|\bbx_i^*-\bbx_j^*\right\|^2-\gamma\right)\geq 0$. 
  Thus it follows that
    \begin{equation}\begin{split}
    \int_0^T \sum_{i=1}^N f_0^i(t,\bbx_i(t)) - f_0^i(t,\bbx^*_i) \, dt 
      \leq \frac{1}{\varepsilon}V_{\bbx^*,\bm{0},\bm{0}}\left(\bbx(0),\bm{\lambda}(0), \bm{\mu}(0)\right).
      \end{split}
    \end{equation}
    Notice that the difference in the left hand side of the previous expression is equal to the first term in \eqref{eqn_first_regret_bound}. Thus, the local regret in \eqref{eqn_first_regret_bound} can be upper bounded by 
  \begin{equation}\label{eqn_second_regret_bound}
    \begin{split}
    \ccalR_T^j &\leq \frac{1}{\varepsilon}V_{\bbx^*,\bm{0},\bm{0}}\left(\bbx(0),\bm{\lambda}(0), \bm{\mu}(0)\right)\\ 
    &+\int_0^T \sum_{i=1, i\neq j}^N f_0^i(t,\bbx_j(t)) - f_0^i(t,\bbx_i(t)) \, dt.
    \end{split}
    \end{equation}
  To bound the second term in the previous expression, use the Lipschitz continuity of the objective function (c.f. Assumption \ref{assumption_convexity})
  \begin{equation}
    \begin{split}
      \left|\int_0^T \sum_{i=1, i\neq j}^N f_0^i(t,\bbx_j(t)) - f_0^i(t,\bbx_i(t)) \, dt\right| \\
      \leq \sum_{i=1,i\neq j}^N \int_0^T L_0\left\|\bbx_j(t)-\bbx_i(t)\right\| \, dt.
\end{split}
  \end{equation}
  The latter is bounded by a function of the order of $\sqrt{T}$ as a consequence of the sublinear network disagreement established in Proposition \ref{proposition_disagreement}. Since $V_{\bbx^*,\bm{0},\bm{0}}\left(\bbx(0),\bm{\lambda}(0), \bm{\mu}(0)\right)/\varepsilon$ is a constant, it holds that \eqref{eqn_second_regret_bound} is upper bounded by a function of the order of $\sqrt{T}$. To complete the proof of the bound in \eqref{eqn_bound_regret} it suffices to replace in the previous expression the network disagreement by the result of Proposition \ref{proposition_disagreement}.
  \end{proof}
We have established that agents operating distributedly achieve sublinear network disagreement, fit and regret. In the next section we discuss the case where the constraint is lower bounded and therefore it cannot be satisfied with slack. This prevents the fit to be bounded due to trajectories that alternate between feasibility and large periods of infeasibility.

\subsection{Saturated Fit}\label{sec_saturated_fit}
We start by definig a saturated function to prevent the constraint to take negative values smaller than a given threshold. Formally, let $\delta>0$ and define the function $\tilde{f}_\delta(t,\bbx) = \max\left\{f(t,\bbx),-\delta \right\}$. Then we can define the notion of saturated local fit as 
\begin{equation}
\tilde{\ccalF}^{ij}_T = \int_0^T \tilde{f}_{\delta,i}(t,\bbx_j(t))\,dt,
\end{equation}
By taking small values of $\delta$ we can arbitrarily reduce the negative portion of the fit. Ideally one would like to set $\delta=0$.
 We next establish that the sublinear bound for the fit established in Theorem \ref{theo_main} holds as well for the saturated fit.%
%
\begin{corollary}\label{coro_feasibility}
    Let Assumptions \ref{assumption_network}--\ref{assumption_bound} hold. Then for any $T\geq 0$ the solutions of the dynamical system \eqref{eqn_gradient_descent}--\eqref{eqn_gradient_ascent}, with $\varepsilon>1/2$, are such that the $k$-th component of the saturated fit $\tilde{\ccalF}_{\delta,Tj}^{i}$ for any $i,j\in\ccalV$ is bounded by
    %
$\left(\tilde{\ccalF}_{\delta,Tj}^{i}\right)_k \leq O(\sqrt{T})$.
\end{corollary}
\begin{proof}
Recall that $\tilde{f}_{\delta,i}(t,\bbx) = \max \left\{f_i(t,\bbx),-\delta \right\}$. Because $\tilde{f}_{\delta,i}(t,\bbx)$ is the point-wise maximum of two convex functions in $\bbx$ it is also convex. Thus, it satisfies the hypotheses of Theorem \ref{theo_main} and it follows that the local saturated fit is also bounded by a function of $O(\sqrt{T})$.
\end{proof}
In the next section we present numerical examples that support the theoretical conclusions.Before doing so, we discuss the bound on the network disagreement established in Proposition \ref{proposition_disagreement}.

\begin{remark}
  In Proposition \ref{proposition_disagreement} we showed that the network disagreement depends on $T$ as $O(\sqrt{T})$. However, the bound achieved is a direct consequence of the choice of the relaxation of the consensus constraint. This choice being arbitrary, it is possible to chose different relaxations to obtain different bounds on the network disagreement. If one desires to bound this quantity by a sublinear function $h(T)$, it suffices to impose the constraint $g(\left\|\bbx_i-\bbx_j\right\|)-\gamma<0$ for any $g(y)=h^{-1}(y)$ as long as $g(y)$ is a convex function. The latter holds because the main component of the proof of the proposition is Jensen's inequality.  
  \end{remark}

\section{Numerical Examples}\label{sec_numerical}
%
  %
In this section we consider a team of $N$ robots tasked with classifying in real time and in a distributed manner the different objects and terrains that compose the environment in which they are deployed. This problem, has been studied in \cite{koppel2015d4l}, although the method presented here is different. Each robot has only access to information about the environment based on the path it has traversed and the images gathered. Therefore, its local information may not be enough to achieve the task of classification since the information gathered may omit regions of the feature space that are crucial. See for instance Figure \ref{fig_road_small} where we depict random trajectories of twenty agents driving around an intersection. When the agent is on the pavement i.e., the absolute value of its horizontal or vertical coordinate is less than five, then it observes pavement images. On the other hand, outside this region it observes grass. As it can be observed in that figure only some of the agents visit both regions and the interest is that the whole team can learn a common classifier. The advantage of learning such classifier is that a robot can identify if it is on grass even if it has not seen grass in the training process. In particular we consider a problem in which each robot receives features $\bbz_i(t)\in \mathbb{R}^n$ from the scene and corresponding labels $y_i(t)\in \left\{-1,1\right\}$ depending on whether the terrain is grass or pavement. The details of the feature extraction from image data is provided in Section \ref{sec_real}. The common objective of the agents can be formulated as training a common linear classifier $\bbx\in\mathbb{R}^n$ that minimizes a loss function. The loss function is such that  its value is small when the classification is accurate and it takes large values in the opposite case. In particular for this problem we consider logistic regression
\begin{equation}\label{eqn_constraint}
 f_i(t,\bbx) = \log\left(1+e^{-y_i(t)\bbx^\top\bbz_i(t)}\right). 
  \end{equation}
The classifier is designed so that the prediction is defined by the sign of the inner product between the classifier $\bbx$ and the feature vector $\bbz_i(t)$ observed by robot $i$ at time $t$.  This is, the predicted label is given by $\hat{y}_i(t) = \sign(\bbx^\top\bbz_i(t))$. Notice that if the prediction is correct, both $\hat{y}_i(t)$ and $y_i(t)$ have the same sign and thus, the exponential in \eqref{eqn_constraint} takes a small value. Which in turn results in $f_i(t,\bbx)$ being small. On the other hand, if the classification is incorrect, the sign of the exponential is positive, which results in a large value of $f_i(t,\bbx)$. Hence, the expression in \eqref{eqn_constraint} is a surrogate of the error function since it results in small values when the prediction is correct and on large values on the other hand.  Notice that agents need to exchange their current actions $\bbx(t)$ with their neighbors to solve the minimization of \eqref{eqn_constraint} using the algorithm defined by \eqref{eqn_gradient_descent} and \eqref{eqn_gradient_ascent}. Since the dimensionality of the actions is as large as the feature vector we want to find a sparse classifier in order to reduce the communication cost. A way of doing so is to include a $\ell_1$ norm regularization in the cost. This regularizer is known to promote sparsity. Let, $\alpha>0$ and define the following local cost 
\begin{figure}
    \centering
    \includegraphics[width=\linewidth,height=0.62\linewidth]{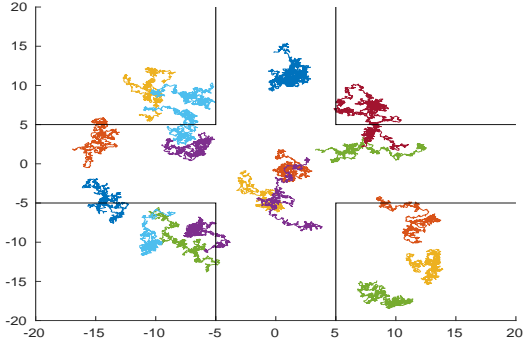}
    \caption{Example of $20$ robots driving randomly at an intersection. When the robot is on the street it is observing pavement images, whereas when it outside of the intersection it has acces to grass images.  }
        \label{fig_road_small}
    \end{figure}
    ~ 
%
%
\begin{equation}
\tilde{f}_i(t,\bbx) = f_i(t,\bbx) + \alpha \left\|\bbx\right\|_1.
  \end{equation}
  The previous objective introduces a tradeoff between classification performance and sparsity. Instead, one can define a desired tolerance for classification error -- by imposing that $f_i(t,\bbx)$ is smaller than a given tolerance $\delta>0$ for all $i=1\ldots N$-- and by minimizing the objective $\|\bbx\|_1$, so to get the sparsest of the solutions. With this idea we define the following centralized problem 
\begin{equation}
\begin{split}
&  \min_{\bbx} \quad\|\bbx\|_1 \\ 
&\,\mbox{s.t.}\quad\, f_i(t,\bbx)-\delta<0 \quad \forall i=1\ldots N.
  \end{split}
  \end{equation}
To solve this problem in a distributed manner, we define -- as done in Section \ref{sec_problem_formulation} -- local copies of the classifier $\bbx_i$ for each agent. The decentralized version of the previous problem then yields 
\begin{equation}\label{eqn_problem_simulations}
\begin{split}
&  \min_{\bbx_1,\bbx_2, \ldots \bbx_N} \,\frac{1}{N}\sum_{i=1}^N\|\bbx_i\|_1 \\ 
  &\quad\quad \mbox{s.t.}\quad f_i(t,\bbx_i)-\delta<0 \quad \forall i=1\ldots N\\
  & \quad \quad \quad \quad \left\|\bbx_i - \bbx_j\right\|^2-\gamma \quad \forall i=1\ldots N.
  \end{split}
  \end{equation}
We evaluate the performance of the saddle-point algorithm \eqref{eqn_gradient_descent}--\eqref{eqn_gradient_ascent} by solving the problem \eqref{eqn_problem_simulations} applied to the team of robots navigating around the intersection depicted in Figure \ref{fig_road_small}. The positions of the $N$ agents is initialized by drawing it from a uniform distribution on the square $[-L,L]^2$ and their paths are random walks updated every $T_s$ seconds, where each step is drawn from a two-dimensional Gaussian variable, with zero mean and covariance matrix $\diag(\sigma_w, \sigma_w)$. Ever $T_s$ seconds each agent has observed $I$ images in the IRA\footnote{Integrated Research Assessment for the U.S. ArmyÕs
Robotics Collaborative Technology Alliance} database \cite{koppel2015d4l} of either grass or pavement. Do notice that even though the algorithm proposed is derived in continuous time, for this application we propose to work with a discrete time system. In Section \ref{sec_results} we present the results achieved by the saddle-point algorithm in the previously described problem. Before doing so, we describe in the next section the feature extraction from the images.
%
\subsection{Data from image database}\label{sec_real}
  The feature extraction is done as in \cite{koppel2015d4l}, a procedure inspired in the two-dimensional texton \cite{leung2001representing}. We describe it next for completeness. The texture features $\bbz_i(t)$ are generated as the sum of a sparse dictionary representation of subpatches of size 24-by-24. This is, each robot classifies images patches of size 24-by-24 by first extracting the nine non-overlapping 8-by-8 sub-patches within it. Each sub-patch is then vectorized, the sample mean subtracted off and divided by its norm. Such that the resulting sub-patch $j$ observed by agent $i$, yields a zero-mean vector $\bbz_i^j$ with norm one. The $9$ vectors resulting from each sub-patch are stacked as columns in a matrix $\bbZ_i = [\bbz_i^1; \ldots; \bbz_i^9]$. On the other hand, the agents have a dictionary of textures that has been trained offline following \cite{koppel2015d4l}. An example of this dictionary can be observed in Figure \ref{fig_dictionary}. The dictionary can be represented by a matrix $\bbD\in\ccalM^{n\times 64}$, where $n$ is the number of features that one wants to extract. The feature used for classification by agent $i$ is the aggregate sparse coding $\bbz_i(t)$, defined as $\bbz_i(t) = \sum_{j=1}^9\bm{\bbz}^*(\bbD;\bbz_i^j(t))$, where $\bbz^*(\bbD;\bbz_i^j(t))$ is the solution to the following optimization problem.
\begin{equation}
\bm{\bbz}^*(\bbD;\bbz_i^j(t)) = \argmin_{\bbz} \frac{1}{2} \| \bbz_i^j(t) - \bbD \bbz\|_2^2 +\zeta\|\bbz\|_1,
  \end{equation}
where $\zeta>0$ the coefficient of the regularization.
%
\subsection{Results}\label{sec_results}
In this section we present the bahavior of the Online Distributed Online Algorithm \eqref{eqn_gradient_descent}--\eqref{eqn_gradient_ascent} for a team of robots that drive in the intersection as the one depicted in Figure \ref{fig_road_small}. For this particular example we consider $N=20$ agents, $L= 15$, $\sigma_w$ and $T_s=1 $ . The parameters of the feature extraction are set to $\zeta=0.125$, $n=128$. We chose $\delta = 0.001$, $\gamma=10$ and the algorithm step size to be $\eta=0.02$, and we consider that each agent has acces to $24$ images per sampling period, in this case, $24$ images per second. 

As it can be observed in Figure \ref{fig_real_disagreement} the network disagreement converges to zero in approximately $6$ seconds, which implies consensus among the agents. This observation supports the theoretical result in Proposition  \ref{proposition_disagreement}. In Figure \ref{fig_real_fit} we depict the network fit for one randomly selected node. As predicted by Theorem \ref{theo_main} the fit is sublinear. The effectiveness of the algorithm can be observed in the classification accuracy achieved by the agents in Figure \ref{fig_real_accuracy}. Notice that the classification error of all the agents is bellow $30\%$. It can be observed as well, that some agents classify with accuracy above $90\%$. The latter is the case for agents that are observing grass. There seems to be an intrinsic difficulty in classifying pavement in the current data set. To support this claim we compute the covariance matrix of $512$ features of images selected randomly. We then project the $192$-dimensional feature vector onto the first two principal components. This projection is depicted in Figure \ref{fig_classes}. As it can be observed the points corresponding to pavement cannot be separated from points corresponding to grass, yet there is a cluster of half of the grass points that is away from the points containing pavement. This suggests that it is indeed harder to classify the pavement images. 

\begin{figure}
  \centering
        \includegraphics[width=0.8\linewidth]{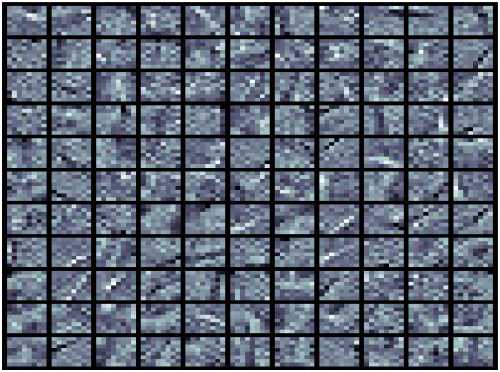}
        \caption{Example of dictionary for 8-by-8 gray scale patches.}
   \label{fig_dictionary}
    \end{figure}
%
\begin{figure}
    \centering
        \includegraphics[width=\linewidth]{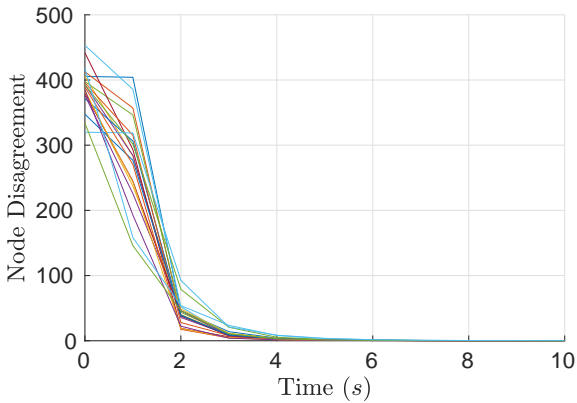}
        \caption{Network disagreement per node for a network of $20$ agents that follow the dynamics \eqref{eqn_gradient_descent}--\eqref{eqn_gradient_ascent}. The feature vectors $\bbz_i(t)\in\mathbb{R}^n$ are extracted from images of the IRA texture database as described in section \ref{sec_real}. The disagreement is sublinear as expected by virtue of Proposition \ref{proposition_disagreement}.}
        \label{fig_real_disagreement}
    \end{figure}
    ~ 
\begin{figure}
\centering
        \includegraphics[width=\linewidth]{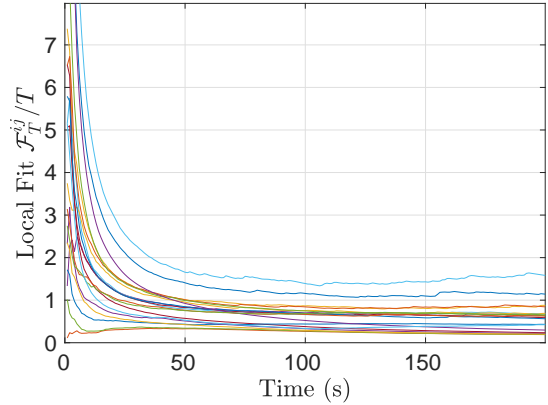}
        \caption{Local fit $\ccalF_T^{j}/T$ for a random node in a network of $N=20$ agents that follow the dynamics \eqref{eqn_gradient_descent}--\eqref{eqn_gradient_ascent}. The feature vectors $\bbz_i(t)\in\mathbb{R}^n$ are extracted from images of the IRA texture database as described in Section \ref{sec_real}. As predicted by Theorem \ref{theo_main} the fit is sublinear.  }
   \label{fig_real_fit}
    \end{figure}
\begin{figure}
  \centering
        \includegraphics[width=\linewidth]{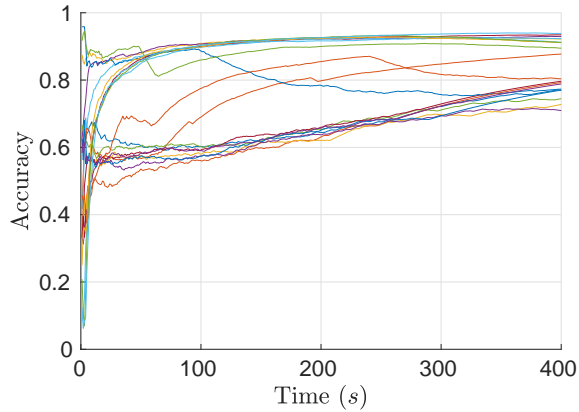}
        \caption{The accuracy of the prediction per node reaches a minimum of $70\%$ for a network of $N=20$ agents that follow the dynamics \eqref{eqn_gradient_descent}--\eqref{eqn_gradient_ascent}. The feature vectors $\bbz_i(t)\in\mathbb{R}^n$ are extracted from images of the IRA texture database as described in Section \ref{sec_real}. As it can be observed some agents achieve accuracy of $90\%$. This agents are observing grass images, whereas those that perform worst are classifying pavement. The fact that one of the classes is poorly classified can be understood as not having a cluster of points where there is no grass as it can be seen in the study of the two principal components of the data set in Figure \ref{fig_classes}.  }
        \label{fig_real_accuracy}
        \end{figure}

        \begin{figure}
  \centering
        \includegraphics[width=\linewidth]{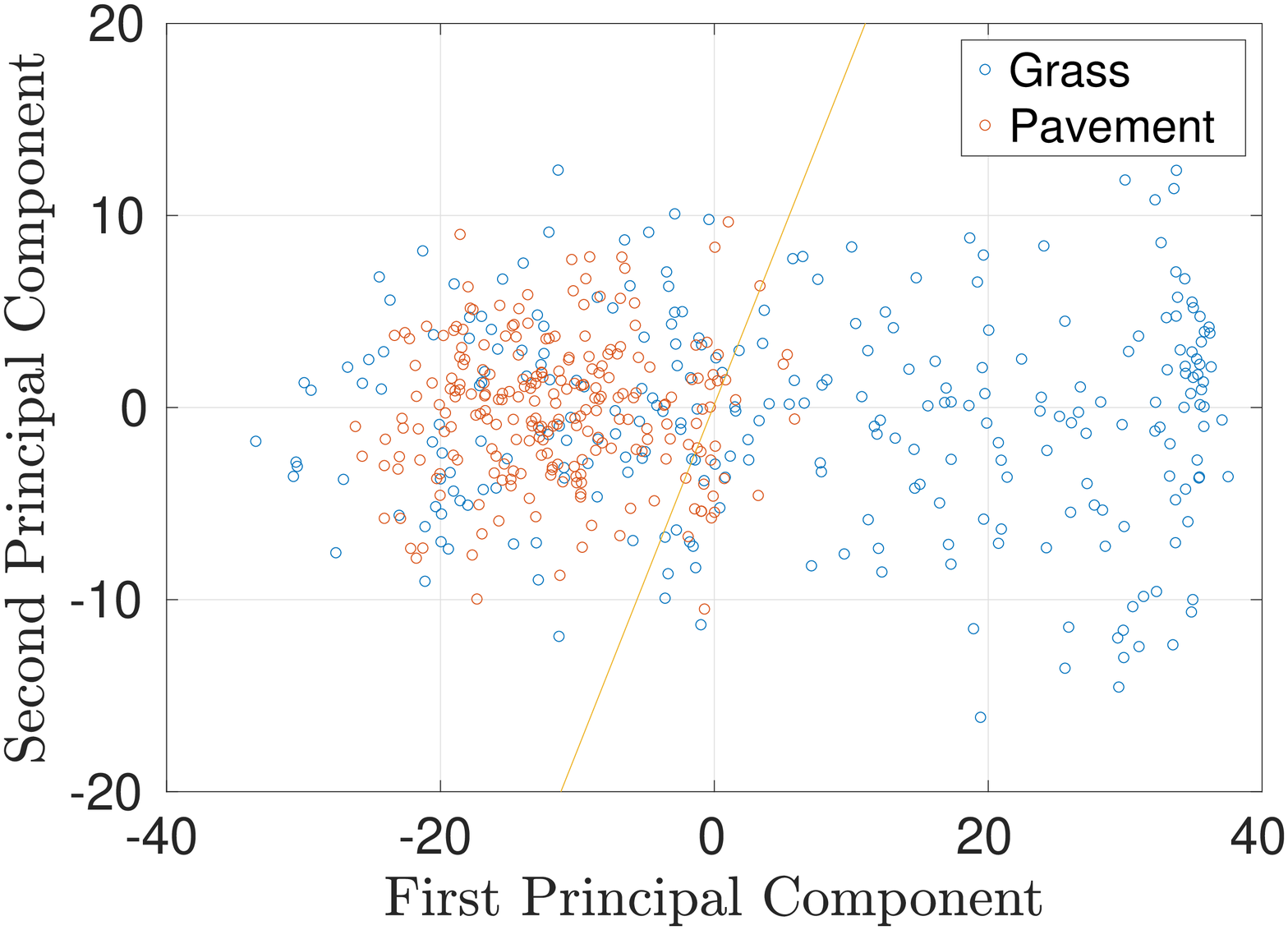}
        \caption{We depict the projection of the features extracted from $512$ images onto the two principal components of the data set. As it can be observed in this picture the images of grass are easier to classify since they present a distinct cluster without any pavement images. On the other hand, the cluster of points corresponding to pavement are intertwined with grass images, which makes its classification harder. We depict as well the projection of the classifier trained by node $1$ after $400$ seconds.  }
        \label{fig_classes}
        \end{figure}

%
    ~ 



\section{Conclusion}\label{sec_conclusion}
We considered the problem of constrained distributed online learning. Each agent only has access to its local constraints and objective function, and the aim is to coordinate the actions among the agents such that the resulting trajectories are feasible and optimal for the team as a whole. We showed that a distributed online version of the saddle point algorithm achieves global fit, regret and network disagreement bounded by functions whose growth rate is bounded by $\sqrt{T}$. The latter result suggests vanishing constraint violation, optimality and network agreement in average as time evolves. We evaluate the performance of the algorithm for a team of robots driving through an urban environment to perform real time texture classification for the purpose of terrain recognition.

\appendix
\section{Appendices}
\subsection{Proof of Lemma \ref{lemma_key_lemma}}\label{ap_key_lemma}
  Let us start by considering the time derivative of the energy function defined in \eqref{eqn_energy_function}. Using the chain rule yields 
  \begin{equation}
    \begin{split}
      \dot{V}_{\tilde{\bbx},\tilde{\bm{\lambda}},\tilde{\bm{\mu}}}(\bbx(t),\bm{\lambda}(t),\bm{\mu}(t)) = \sum_{i=1}^N \left(\bbx_i(t)-\tilde{\bbx}_i\right)^\top\dot{\bbx_i}(t) \\
           + \sum_{i=1}^N\left(\bm{\lambda}_i(t)-\tilde{\bm{\lambda}}_i\right)^\top\dot{\bm{\lambda}}_i(t) +\sum_{i=1}^N\left(\bm{\mu}_i(t)-\tilde{\bm{\mu}}_i\right)^\top\dot{\bm{\mu}}_i(t). 
    \end{split}
    \end{equation}
Substituting the time derivatives by those given by the Distributed Online Saddle Point dynamics \eqref{eqn_gradient_descent}-\eqref{eqn_gradient_ascent} in the previous expression yields
\begin{equation}\label{eqn_v_dot0}
  \begin{split}
    &\dot{V}_{\tilde{\bbx},\tilde{\bm{\lambda}},\tilde{\bm{\mu}}}(\bbx(t),\bm{\lambda}(t),\bm{\mu}(t)) \\
    &= \sum_{i=1}^N \left(\bbx_i(t)-\tilde{\bbx}_i\right)^\top\Pi_{\ccalX}\left[\bbx_i(t),-\varepsilon\ccalL_{x_i}(t,\bbx(t),\bm{\lambda}(t),\bm{\mu}(t))\right]\\
&+ \sum_{i=1}^N\left(\bm{\lambda}_i(t)-\tilde{\bm{\lambda}}_i\right)^\top\Pi_+\left[\bm{\lambda}_i(t),\varepsilon\ccalL_{\lambda_i}(t,\bbx_i(t),\bm{\lambda}_i(t),\bm{\mu}_i(t))\right]\\
&+\sum_{i=1}^N\left(\bm{\mu}_i(t)-\tilde{\bm{\mu}}_i\right)^\top\Pi_{+}\left[\bm{\mu}(t),\varepsilon\ccalL_\mu^i(t,\bbx_i(t),\bm{\lambda}_i(t),\bm{\mu}_i(t))\right]. 
    \end{split}
    \end{equation}
Since both $\bbx_i(t)$ and $\tilde{\bbx}_i$ belong to the convex set $\ccalX$, the inner product between $\bbx_i(t)-\tilde{\bbx}_i$ and the projected vector can be upper bounded by the the product with the field without the projection (cf., Lemma 1 \cite{PaternainRibeiro16})
\begin{equation}
  \begin{split}
    \left(\bbx_i(t)-\tilde{\bbx}_i\right)^\top\Pi_{\ccalX}\left[\bbx_i(t),-\varepsilon\ccalL_{x_i}(t,\bbx(t),\bm{\lambda}(t),\bm{\mu}(t))\right]\\
    \leq -\left(\bbx_i(t)-\tilde{\bbx}_i\right)^\top\varepsilon\ccalL_{x_i}(t,\bbx(t),\bm{\lambda}(t),\bm{\mu}(t)).
\end{split}
  \end{equation}
Summing across agents on both sides of the previous expression allows us to upper bound the first term summation in \eqref{eqn_v_dot0}
\begin{equation}\label{eqn_bound_x}
  \begin{split}
    \sum_{i=1}^N\left(\bbx_i(t)-\tilde{\bbx}_i\right)^\top\Pi_{\ccalX}\left[\bbx_i(t),-\varepsilon\ccalL_{x_i}(t,\bbx(t),\bm{\lambda}(t),\bm{\mu}(t))\right]\\
        \leq -\varepsilon     \sum_{i=1}^N \left(\bbx(t)-\tilde{\bbx}\right)^\top\ccalL_{x}(t,\bbx(t),\bm{\lambda}(t),\bm{\mu}(t)).
\end{split}
\end{equation}
In addition, since the Lagrangian is a convex function with respect to $\bbx_i(t)$ we can further upper bound the sum of inner products by the difference between the Lagrangian evaluated at $\tilde{\bbx}_i$ and $\bbx_i(t)$. Proceeding in this way for all $i\in\ccalV$ yields
\begin{equation}\label{eqn_bound_x}
  \begin{split}
    \sum_{i=1}^N\left(\bbx_i(t)-\tilde{\bbx}_i\right)^\top\Pi_{\ccalX}\left[\bbx_i(t),-\varepsilon\ccalL_{x_i}(t,\bbx(t),\bm{\lambda}(t),\bm{\mu}(t))\right]\\
    \leq \varepsilon\left(\ccalL(t,\tilde{\bbx},\bm{\lambda}(t),\bm{\mu}(t))-\ccalL(t,\bbx(t),\bm{\lambda}(t),\bm{\mu}(t))\right).
\end{split}
  \end{equation}
Likewise, for the multipliers the following relationships hold by virtue of \cite[Lemma 1]{PaternainRibeiro16}
\begin{equation}
  \begin{split}
   \left(\bm{\lambda}_i(t)-\tilde{\bm{\lambda}}_i\right)^\top\Pi_+\left[\bm{\lambda}_i(t),\varepsilon\ccalL_{\lambda_i}(t,\bbx(t),\bm{\lambda}(t),\bm{\mu}(t))\right]\\
    \leq \left(\bm{\lambda}_i(t)-\tilde{\bm{\lambda}}_i\right)^\top\varepsilon\ccalL_{\lambda_i}(t,\bbx(t),\bm{\lambda}(t),\bm{\mu}(t)),
\end{split}
  \end{equation}
and
\begin{equation}
  \begin{split}
   \left(\bm{\mu}_i(t)-\tilde{\bm{\mu}}_i\right)^\top\Pi_+\left[\bm{\mu}_i(t),\varepsilon\ccalL_{\mu_i}(t,\bbx(t),\bm{\lambda}(t),\bm{\mu}(t))\right]\\
    \leq \left(\bm{\mu}_i(t)-\tilde{\bm{\mu}}_i\right)^\top\varepsilon\ccalL_{\mu_i}(t,\bbx(t),\bm{\lambda}(t),\bm{\mu}(t)).
\end{split}
  \end{equation}
Because the Lagrangian is linear with respect to $\bm{\lambda}$ and $\bm{\mu}$ (cf., \eqref{eqn_lagrangian}) we can write the above inner products as differences of the Lagrangian evaluated at $\bm{\lambda}_i(t)$ and $\bar{\bm{\lambda}}_i$ and as differences of the Lagrangian evaluated at $\bm{\mu}_i(t)$ and $\bar{\bm{\mu}}_i$
\begin{equation}\label{eqn_bound_multipliers}
  \begin{split}
\sum_{i=1}^N \left(\bm{\lambda}_i(t)-\tilde{\bm{\lambda}}_i\right)^\top\Pi_+\left[\bm{\lambda}_i(t),\varepsilon\ccalL_{\lambda_i}(t,\bbx(t),\bm{\lambda}(t),\bm{\mu}(t))\right]\\
    +\sum_{i=1}^N \left(\bm{\mu}_i(t)-\tilde{\bm{\mu}}_i\right)^\top\left[\bbmu_i(t),\varepsilon\ccalL_{\mu_i}(t,\bbx(t),\bm{\lambda}(t),\bm{\mu}(t))\right] \\
    \leq \varepsilon\left(\ccalL(t,\bbx(t),\bm{\lambda}(t),\bm{\mu}(t))-\ccalL(t,\bbx(t),\tilde{\bm{\lambda}},\tilde{\bm{\mu}})\right).
\end{split}
\end{equation}
Substituting the expressions \eqref{eqn_bound_x} and \eqref{eqn_bound_multipliers}  in \eqref{eqn_v_dot0} reduces to 
\begin{equation}\label{eqn_v_dot}
  \begin{split}
    \dot{V}_{\tilde{\bbx},\tilde{\bm{\lambda}},\tilde{\bm{\mu}}}(\bbx(t),\bm{\lambda}(t),\bm{\mu}(t)) \\
    \leq \varepsilon\left(\ccalL(t,\tilde{\bbx},\bm{\lambda}(t),\bm{\mu}(t))-\ccalL(t,\bbx(t),\bm{\lambda}(t),\bm{\mu}(t))\right) \\
    +\varepsilon\left(\ccalL(t,\bbx(t),\bm{\lambda}(t),\bm{\mu}(t))-\ccalL(t,\bbx(t),\tilde{\bm{\lambda}},\tilde{\bm{\mu}})\right). \\
  \end{split}
    \end{equation}
Observe that the second and third terms on the right hand side of the previous expression cancel out for all $t$. Hence, the previous bound reduces to
\begin{equation}
\begin{split}
    \dot{V}_{\tilde{\bbx},\tilde{\bm{\lambda}},\tilde{\bm{\mu}}}(\bbx(t),\bm{\lambda}(t),\bm{\mu}(t)) \leq \\\varepsilon\left(\ccalL(t,\tilde{\bbx},\bm{\lambda}(t),\bm{\mu}(t))-\ccalL(t,\bbx(t),\tilde{\bm{\lambda}},\tilde{\bm{\mu}})\right).
    \end{split}
\end{equation}
 Rearranging the terms in the previous inequality and integrating both sides from $t=0$ until the time horizon $t=T$ yields
\begin{equation}
  \begin{split}
    \int_0^T\ccalL(t,\bbx(t),\tilde{\bm{\lambda}},\tilde{\bm{\mu}})-\ccalL(t,\tilde{\bbx},\bm{\lambda}(t),\bm{\mu}(t))\, dt\\
    \leq  -\frac{1}{\varepsilon}\int_0^T\dot{V}_{\tilde{\bbx},\tilde{\bm{\lambda}},\tilde{\bm{\mu}}}(\bbx(t),\bm{\lambda}(t),\bm{\mu}(t))\, dt.
  \end{split}
    \end{equation}
By virtue of the Fundamental Theorem of Calculus, the right hand side of the previous expression reduces to the difference between ${V}_{\tilde{\bbx},\tilde{\bm{\lambda}},\tilde{\bm{\mu}}}(\bbx(t),\bm{\lambda}(t),\bm{\mu}(t))$ evaluated at time $T$ and $0$.
%
The proof if completed by observing that for any $T$, ${V}_{\tilde{\bbx},\tilde{\bm{\lambda}},\tilde{\bm{\mu}}}(\bbx(T),\bm{\lambda}(T),\bm{\mu}(T))$ is non-negative and thus, the right hand side can be upper bounded by the energy function evaluated at $t=0$.


\bibliographystyle{ieeetr}
\bibliography{bib}

\end{document}